\documentclass[12pt]{amsart}

\usepackage{epsfig,psfrag}
\usepackage{amssymb, latexsym}
\usepackage[all]{xy}
\usepackage{amssymb}
\usepackage{amsthm}
\usepackage{amsmath}
\usepackage{amsxtra}

\newtheorem{theorem}{Theorem}[section]

\newtheorem{lemma}[theorem]{Lemma}
\newtheorem{prop}[theorem]{Proposition}

\newtheorem{ex}[theorem]{Example}

\theoremstyle{definition}
\newtheorem{defn}[theorem]{Definition}
\newtheorem*{notation}{Notation}
\newtheorem*{rmk}{Remark}

\numberwithin{equation}{section}

\newcommand{\Q}{\mathbb Q}
\newcommand{\R}{\mathbb R}
\newcommand{\F}{\mathbb F}
\newcommand{\C}{\mathbb C}

\newcommand{\Z}{\mathbb Z}

\begin{document}
\baselineskip=17pt 
\title{Counting points over finite fields and hypergeometric functions}
\author{Adriana Salerno}
\address{Bates College\\
3 Andrews Road\\
Lewiston, ME 04240\\ 
U.S.A.}
\email{asalerno@bates.edu}

\begin{abstract} It is a well known result that the number of points over a finite field on the Legendre family of elliptic curves can be written in terms of a hypergeometric function modulo $p$. In this paper, we extend this result, due to Igusa, to a family of monomial deformations of a diagonal hypersurface. We find explicit relationships between the number of points and generalized hypergeometric functions as well as their  finite field analogues. 
\end{abstract}

\keywords{counting rational points over a finite field, hypergeometric functions}
\subjclass[2010]{11G25, 33C20, 14G05 }

\maketitle

\section{Introduction}

For each $\lambda\in\mathbb{P}^1-\{0,1,\infty\}$ we can define an elliptic curve

\[E_{\lambda}:y^2=x(x-1)(x-\lambda).\]

These form the so-called Legendre family. There is a classical result by Igusa \cite{igusa} that states that for $\lambda \in \Z$ the number of $\mathbb{F}_p$-points on these curves, $N_{\F_p}(\lambda)$, is a hypergeometric function of the parameter $\lambda$ (modulo $p$). In fact, a simple computation (cf. \cite{clemens}) shows that

 \[N_{\F_p}(\lambda)\equiv(-1)^{(p+1)/2}\sum_{r=0}^{\frac{p-1}{2}}{-1/2\choose r}^2\lambda^r\bmod p\]
 \[\equiv(-1)^{(p-1)/2}\sum_{r=0}^{\frac{p-1}{2}}\frac{(1/2)_r(1/2)_r}{r!r!}\lambda^r\bmod p,\] where the last sum is the hypergeometric function $$ {}_2F_1\left(\left.\frac{1}{2},\frac{1}{2};1\right|\lambda\right)$$ truncated at $\frac{p-1}{2}$. 

We believe we can find results like Igusa's in general, that is, that hypergeometric functions should appear in some capacity when counting $\F_q$-rational points.  In this paper, we study the relationship between the number of $\F_q$-rational points and hypergeometric functions for a family of \emph{monomial deformations of diagonal hypersurfaces}. These families are of the form:

\begin{equation}
X_{\lambda}:x_1^d+\cdots+x_n^d-d\lambda x_1^{h_1}\cdots x_n^{h_n}=0
\label{E:monomial}
\end{equation} where $\sum h_i=d, \text{g.c.d.}(d,h_1,\dots,h_n)=1$. For $\lambda \in \Z$, let $N_{\F_q}(\lambda)$ denote the number of points on the hypersurface in $\mathbb{P}^{n-1}_{\F_q}$.

 We have explored the relationship between $N_{\F_q}(\lambda)$ and hypergeometric functions in two ways. The basis for our approach in both cases is an important result by Koblitz \cite{koblitz} (Theorem \ref{T:koblitz} in this paper). 

First, we use Koblitz's formula to relate $N_{\F_q}(\lambda)$ to the finite field version of a hypergeometric function as defined by Katz \cite{katz:esde}. In this case, we see that the number of $\F_q$-points cannot be written in terms of a single hypergeometric function, but as a sum of several hypergeometric functions. This is the content of Theorem \ref{T:ff}. 

We will then use Koblit's result and the Gross-Koblitz formula to find an explicit relationship between $N_{\F_p}(\lambda)$ (restricting our attention to fields of prime order) and generalized hypergeometric functions in some special cases. 
The first case is a zero-dimensional variety, and the surprising result (Theorem \ref{T:mine}) is that even in this simple case there are many hypergeometric functions that appear. In the second case, we look at a known computation, the famous Dwork family, using our methods and see that it behaves much more like the Legendre family. 

It is worth noting that there are many other ways in which one can approach this problem. For example, in \cite{kloost}, Kloosterman computes the Zeta function (and in consequence the generating function for $N_{\F_q}(\lambda)$) using the geometry of these hypersurfaces, and finds a relationship with hypergeometric functions. This geometric approach was also used, to a certain extent,  in   \cite{salerno}. Lennon, in her thesis \cite{lennon}, related elliptic curves to Greene's finite field version of hypergeometric functions. This elliptic curves approach has led to the study of so-called hypergeometric modular forms, as in Ono and Mahlburg's work \cite{ono}. 

Our goal with our particular  approach was to emulate the simple computation of Igusa's to find this relationship, and we were surprised to find how much more difficult the calculation becomes when going outside of the Dwork family and the Legendre family examples.

\section{Background}
The series

\[\sum_{k\geq 0} \frac{(a)_k(b)_k}{(c)_k}\frac{z^k}{k!},\] where we use the Pochhammer notation

\[(x)_k=x(x+1)\cdots(x+k-1)=\frac{\Gamma(x+k)}{\Gamma(x)},\] is called the \emph{Gauss hypergeometric function}.

    Many variants of the definition of a hypergeometric function have arisen since Gauss first defined it, a few of which will be used throughout this work. In this section, we will introduce three versions of this function and present some of their most important features and properties.

\subsection{The generalized hypergeometric function}

The most classical definition is the extension of Gauss's hypergeometric function, with notation due to Barnes, c.f. \cite{slater}.

\begin{defn}

Let, $A,B\in\Z$ and $\alpha_1,\dots,\alpha_A,\beta_1,\dots,\beta_B\in\Q$, with all of the $\beta_i\geq0$. The \emph{generalized hypergeometric function} is defined as the series (taking $z\in\C$)

\begin{equation}{}_AF_B(\alpha_1,\dots,\alpha_A;\beta_1,\dots,\beta_B|z)=\sum_{k=0}^{\infty}\frac{(\alpha_1)_k\cdots(\alpha_A)_kz^k}{(\beta_1)_k\cdots(\beta_B)_kk!}.\label{E:hyperg}\end{equation}

\end{defn}

The $\alpha_i$ will be referred to as ``numerator parameters'' and the $\beta_i$ as ``denominator parameters''.

Notice that in this notation Gauss's hypergeometric function becomes ${}_2F_1(\alpha_1,\alpha_2;\beta_1|z)$.

Sometimes we will use the shortened notation $${}_AF_B(\alpha;\beta|z)={}_AF_B(\alpha_1,\dots,\alpha_A;\beta_1,\dots,\beta_B|z).$$

\subsection{Hypergeometric weight systems}\label{S:weight}

One can think of (\ref{E:hyperg}) in terms of ratios of factorials (rather than Pochhammer symbols). In \cite{frv}, Rodr\'{i}guez-Villegas defines a \emph{hypergeometric weight system} as a formal linear combination

\[\gamma=\sum_{\nu\geq1}\gamma_{\nu}[\nu], \hspace{1cm} \nu\in\Z,\] where the $\gamma_{\nu}\in\Z$ are zero for all but finitely many $\nu$, satisfying the following conditions:
\begin{enumerate}
\item $\sum_{\nu\geq1}\nu\gamma_{\nu}=0$
\item $d=d(\gamma):=-\sum_{\nu\geq1}\gamma_{\nu}>0$
\end{enumerate}

To $\gamma$ we can associate the formal power series
\[u(\lambda):=\sum_{n\geq0}u_n\lambda^n\]
\noindent where
\[u_n=\prod_{\nu\geq1}(\nu n)!^{\gamma_{\nu}}.\]

\begin{lemma}[Rodr\'{i}guez-Villegas] $u$ is a hypergeometric function, that is, for some minimal $r$ we have
\[u(\lambda)={}_rF_{r-1}\left(\alpha_1,\dots,\alpha_r;\beta_1,\dots,\beta_{r-1}\left|\dfrac{\lambda}{\lambda_0}\right.\right)\]
where $\lambda_0^{-1}=\prod_{\nu\geq1}\nu^{\nu\gamma_{\nu}}$ and $\alpha_1,\dots,\alpha_r,\beta_1,\dots,\beta_{r-1}$ are rational numbers.
\end{lemma}

Thus, we can think of a hypergeometric function as being associated to a hypergeometric weight system and viceversa, provided that certain conditions are satisfied.

There is a useful function associated to a hypergeometric weight system which we will now define. 

\begin{defn}
The \emph{Landau function} associated to $\gamma$ is defined by

\[\mathcal{L}(x)=\mathcal{L}_{\gamma}(x):=-\sum_{\nu\geq 1}\gamma_{\nu}\{\nu x\}, \hspace{1cm} x\in\R\] where $\{x\}$ denotes the fractional part of $x$. This function is periodic of period 1.

\end{defn}

The Landau function is useful for checking whether the coefficients of the series $u(z)$ are integers.

\begin{prop}[Landau]
$u_n\in\Z$ for all $n\geq0$ if and only if $\mathcal{L}(x)\geq0$ for all $x\in\R$.
\end{prop}

We want to point out a crucial step of the proof because it will be used later.

\begin{lemma}
Let $p$ be a prime and let $v_p(x)$ denote the $p$-adic valuation of $x$.

 \[v_p(u_n)=\sum_{k\geq 1}\mathcal{L}\left(\frac{n}{p^k}\right)\]

 \end{lemma}

So the Landau function encodes information about the $p$-adic valuation of the coefficients of the series.

This function has many other properties as listed in \cite{frv}. Here we list a few which will be useful in some of our computations later on.

\begin{prop}
\begin{enumerate}

\item $\mathcal{L}$ is right continuous with discontinuity points exactly at $x\equiv\alpha_i\bmod 1$ or $x\equiv\beta_i\bmod 1$ for some $i=1,\dots,r$.
 More precisely,

\[\mathcal{L}=\#\{j|\alpha_i\leq x\}-\#\{j|0<\beta_j\leq x\}.\]

\item $\mathcal{L}$ takes only integer values.

\item Away from the discontinuity points of $\mathcal{L}$ we have

\[\mathcal{L}(-x)=d-\mathcal{L}(x)\]

\noindent and, in particular, for all $x$

\[\mathcal{L}(x)\leq d,\hspace{1cm} \text{if $u_n\in\Z$ for all $n$}.\]

\end{enumerate}
\label{P:Landau}
\end{prop}

\subsection{A finite field analog}\label{S:ff}

First, we establish notation, following \cite{koblitz}. Let $\chi_{1/(q-1)}:\F_q^*\rightarrow K^*$ be a fixed generator of the character group of $\F_q^*$, where $K$ is an algebraically closed field of characteristic zero (such as $\C$ or $\C_p$).

\begin{ex}
\begin{enumerate}
 \item If $K=\C$ fix a primitive root of $\F_q^*$ and define $\chi_{1/(q-1)}$ by taking that root to $e^{2\pi i/(q-1)}$.
\item If $K=\C_p$ we can take $\chi_{1/(q-1)}$ to be the Teichm\"{u}ller character. Recall that $\omega: \F_q^*\rightarrow \C_p^*$ is the Teichm\"{u}ller character where $\omega(x)$ is defined as the unique element of $\C_p^*$ which is a $(q-1)$-st root of unity and such that $\omega(x)\equiv x\bmod p$.

\end{enumerate}
\end{ex}

For $s \in \frac{1}{q-1}\Z/\Z$ we let $\chi_s=\left(\chi_{1/(q-1)}\right)^{s(q-1)}$, and for any $s$ set $\chi_s(0)=0$. Let $\psi:\F_q\rightarrow K^*$ be a (fixed) additive character.

\begin{defn}

For $s\in\frac{1}{(q-1)}\Z/\Z$ we let $g(s)$ denote the Gauss sum

\[g(s)=\sum_{x\in\F_q}\chi_s(x)\psi(x)\]
\end{defn}

\begin{lemma}
 Gauss sums satisfy the following properties:
\begin{enumerate}
 \item $g(s)g(-s)=q\chi_s(-1)$ if $s\neq0$, and $g(0)=-1$.
\item If $d|q-1$,
\[\prod_{j=0}^{d-1}g\left(s+\frac{j}{d}\right)=\chi_{-ds}(d)g(ds)\prod_{j=1}^{d-1}g\left(\frac{j}{d}\right)\]
\end{enumerate}
\label{L:Gauss}
\end{lemma}

For a proof of the lemma see, for example, \cite{ireland}. 

\begin{defn}
If $s_1,\dots,s_r\in\dfrac{1}{q-1}\Z/\Z$ and the sum of the $s_i$'s  is not an integer, we define the Jacobi sum

\[J(s_1,\dots,s_r)=\sum_{\substack{x_1,\dots,x_r\in\F_q\\x_1+\cdots+x_r=1}}\chi_{s_1}(x_1)\cdots\chi_{s_r}(x_r), r>1; J(s_1)=1.\]
\end{defn}

Jacobi sums can be expressed in terms of Gauss sums as follows:

\[J(s_1,\dots,s_r)=\frac{g(s_1)\cdots g(s_r)}{g(s_1+\cdots+s_r)}.\]

A finite field analog of the hypergeometric function was defined by Katz \cite{katz:esde} as follows.

\begin{defn}

Let $t\in \F_q^*$. Define the set
\[V_t=\left\{x\in (\F_q^*)^n, y\in (\F_q^*)^m|x_1\cdots x_n=ty_1\cdots y_m\right\}\]

Also, let $\psi:\F_q\to K^*$, be a (fixed) additive character where $K$ is an algebraically closed field (like $\C$ or $\C_p$), let $\chi$ denote, as in the previous section, a generator of the character group of $\F_q^*$, and $\alpha_1,\dots,\alpha_n,\beta_1,\dots,\beta_m\in\dfrac{1}{q-1}\Z/\Z$ so that $\chi_{\alpha_1},\dots,\chi_{\alpha_n},\chi_{\beta_1},\dots,\chi_{\beta_m}:\F_q^*\to K^*$ are multiplicative characters. Then we define the \emph{finite field version of a hypergeometric function} as

\begin{align*}
H(\alpha;\beta|t):&=\sum_{x,y\in V_t}\psi(x_1+\cdots+x_n-(y_1+\cdots+y_m))\chi_{\alpha_1}(x_1)\cdots\chi_{\alpha_n}(x_n)\\
&\cdot\overline{\chi}_{\beta_1}(y_1)\cdots\overline{\chi}_{\beta_m}(y_m)\\
\end{align*}

\end{defn}

It will be convenient to think of this definition in a different form which is given by its Fourier series expansion.

\begin{lemma}
The Fourier series expansion of $H(\alpha;\beta|t)$ is
\[H(\alpha;\beta|t)=\frac{1}{q-1}\sum g(s+\alpha_1)\cdots g(s+\alpha_n)g(-s-\beta_1)\cdots g(-s-\beta_m)\overline{\chi_s}(t)\] where the sum is taken over $s\in\frac{1}{q-1}\Z/\Z$.
\label{L:katz}
\end{lemma}

\section{Koblitz's formula}

This section summarizes the main results in a paper by Koblitz \cite{koblitz}, in which he gives formulas for the number of points on monomial deformations of diagonal hypersurfaces, in terms of Gauss and Jacobi sums. Much of the work is a generalization of the proofs and ideas in a famous paper by Weil \cite{weil}.

\subsection{Weil's theorem}

Suppose we have an algebraic variety $X$ defined over a finite field $\F_q$ and we want to determine the number $N_{\F_q}(X)$ of $\F_q$-points on it. Notice that these points are the $\overline{\F}_q$-points of $X$ fixed by the $q$-th power Frobenius map $F:(\dots, x_i,\dots)\mapsto(\dots,x_i^q,\dots)$. Thus, we get

\[N_{\F_q}(X)=\#\{x\in X|F(x)=x\}.\]

Suppose we have a group $G$ acting on $X$. Then we can split up $N_{\F_q}(X)$ into pieces $N_{\F_q}(X,\chi)$, where $\chi:G\rightarrow K^*$ is a character as in Section \ref{S:ff} . $N_{\F_q}(V,\chi)$ is thus defined by:

\[N_{\F_q}(X,\chi)=\frac{1}{\# G}\sum_{\xi\in G}\chi^{-1}(\xi)\#\{x\in X| F\circ\xi(x)=x\}.\]

Since in all of our examples $G$ will be abelian, the only irreducible representations will be one-dimensional characters $\chi$. In that case, we have the following lemma, which follows immediately from the previous definitions

\begin{lemma}

\[N_{\F_q}(X)=\sum_{\chi\in char(G)}N_{\F_q}(X,\chi).\]

\end{lemma}

The simplest example of a variety with a large group action is the diagonal hypersurface of degree $d$ in $\mathbb{P}_{\F_q}^{n-1}$ (here $d|q-1$):

\[D_{d,n}:x_1^d+\cdots+x_n^d=0\]

The group $\mu_d^n$ of $n$-tuples of $d$-th roots of unity in $\F_q^*$ acts on $D_{d,n}$ by $\xi=(\xi_1,\dots,\xi_n)$ taking the point $(x_1,\dots,x_n)$ to $(\xi_1x_1,\dots,\xi_nx_n)$. Let $\Delta$ be the diagonal elements of $\mu_d^n$, i.e. elements of the form $(\xi,\cdots,\xi)$. Notice that $\Delta$ acts trivially on $D_{d,n}$ and $\mu_d^n/\Delta$ acts faithfully. The character group of $\mu_d^n/\Delta$ is in one-to-one correspondence with the $n$-tuples

\[w=(w_1,\dots,w_n),0\leq w_i<d,\text{for which}\sum w_i\equiv0\bmod d,\] where

\[\chi_w(\xi):=\chi(\xi^w),\hspace{1cm}\xi^w=\xi_1^{w_1}\cdots\xi_n^{w_n}\] and $\chi$ is a fixed primitive character of $\mu_d$, which we can get for example by restricting $\chi_{1/(q-1)}$ to $\mu_d$. In \cite{weil}, Weil proves:

\begin{theorem}[Weil]

\[N_{\F_q}(D_{d,n},\chi_w)=\left\{\begin{array}{cc}
                                    
                                    \dfrac{q^{n-1}-1}{q-1}&\text{if $w_i=0$, $  \forall i$}\\
                                    -\dfrac{1}{q}J\left(\dfrac{w_1}{d},\dots,\dfrac{w_n}{d}\right)&\text{if $w_i\neq0$, $\forall i$}\\
                                    0&\text{otherwise}\\
                                    \end{array}\right.\]

\end{theorem}

\subsection{Koblitz's formula}

The goal of Koblitz's paper is to use Weil's result and similar methods to find the number of points on the monomial deformation (\ref{E:monomial}). Notice that these hypersurfaces allow an action of the group $$G=\{\xi\in\mu_d^n|\xi^h=1\}/\Delta,$$ consisting of elements which preserve the monomial $x^h=x_1^{h_1}\cdots x_n^{h_n}$.

The characters $\chi_w$ of $\mu_d^n/\Delta$ which act trivially on $G$ are precisely powers of $\chi_h$. Thus, $char(G)$, the character group of $G$, corresponds to equivalence classes of $w$ in $$W=\{(w_1,\dots,w_n)|0\leq w_i<d,\sum w_i\equiv0\mod d\},$$ where $w'\sim w$ if $w-w'$ is a multiple (mod d) of $h$. Notice that, since $g.c.d(d,h_1,\dots,h_n)=1$, each equivalence class contains $d$ $n$-tuples $w'$.

We are now ready to state the main theorem of Koblitz's paper.

Assume $d|q-1$ and let $N_{\F_q}(0)$ be the number of $\F_q$-points on the diagonal hypersurface $D_{d,n}$.

\begin{theorem}[Koblitz]

\[N_{\F_q}(\lambda)=N_{\F_q}(0)+\frac{1}{q-1}\sum_{\substack{s\in \frac{d}{q-1}\Z/\Z\\w\in W}}\frac{g\left(\dfrac{w+sh}{d}\right)}{g(s)}\chi_s(d\lambda),\]

\noindent where we denote $g\left(\dfrac{w+sh}{d}\right)=\displaystyle\prod_ig\left(\dfrac{w_i+sh_i}{d}\right)$.
\label{T:koblitz}
\end{theorem}

\section{Finite field results}
In this section, we will see that $N_{\F_q}(\lambda)-N_{\F_q}(0)$ is related to the finite field version of a hypergeometric function. In \cite{koblitz}, Koblitz shows that the number of points is an analogue of a Barnes type integral, which is in turn analogous to the generalized hypergeometric function. We use the same strategy but with a different endgame, which is to relate Theorem \ref{T:koblitz} directly to Katz's finite field hypergeometric function (as described in Lemma \ref{L:katz}). 

First, Koblitz considers for some fixed $w$ the sum

\[\sum_{\substack{s\in\frac{d}{q-1}\Z/\Z\\w'\sim w}} \frac{g\left(\dfrac{w+sh}{d}\right)}{g(s)}\chi_s(d\lambda)\]

It is not hard to check that if we replace $d$ by $ds$ and sum over $s\in\frac{1}{(q-1)}\Z/\Z$ we obtain

\[\sum_{s\in\frac{1}{q-1}\Z/\Z}\frac{g\left(hs+\dfrac{w}{d}\right)}{g(ds)}\chi_{ds}(d\lambda).\]

Using Lemma \ref{L:Gauss}, one can rewrite the previous statement as

\begin{equation}\prod_{j=1}^{d-1} g\left(\frac{j}{d}\right)\sum_s\frac{g\left(h_1s+\dfrac{w_1}{d}\right)\cdots g\left(h_ns+\dfrac{w_n}{d}\right)}{g(s)g\left(s+\dfrac{1}{d}\right)\cdots g\left(s+\dfrac{d-1}{d}\right)}\chi_{ds}(\lambda).\label{E:calculation}\end{equation}

This is the expression which is analogous to a Barnes type integral and is thus analogous to a hypergeometric function. Since we want explicit formulas, we take this method of computation further. The result is the following theorem. 

\begin{theorem} \label{T:ff} Assume $dh_1h_2\cdots h_n|q-1$. 
\begin{align*} 
&N_{\F_q}(\lambda)-N_{\F_q}(0)=\xi q^{\frac{n-2d-1}{2}}\cdot\\
&\sum_{[w]\in W/\sim}H\left(\left.0,\frac{1}{d},\dots,\frac{d-1}{d};\dots,1-\frac{w_i+dj}{dh_i},\dots\right|\prod_{i=1}^nh_i^{h_i}(-\lambda)^d\right),
\end{align*} where the denominator parameters run through the $h_i$ values $\frac{w_i+dj}{dh_i}, j=0,\dots,h_i-1$ for each $i$, and no exponent appears if $h_j=0$, and modulo cancelation if the numerator and denominator terms are the same. Here $\xi$ is a $q-1$ root of unity.
\end{theorem}

\begin{proof}

Koblitz's computation, described above, gets us to equation (\ref{E:calculation}). For each $h_i$, notice that we can use Lemma \ref{L:Gauss} again, but we need to assume $dh_i|q-1$ for all $i$. Basically, this means that all of our upcoming computations will make sense for a large enough $q$.

\begin{align*}
g\left(h_is+\frac{w_i}{d}\right)&=g\left(h_i\left(s+\frac{w_i}{dh_i}\right)\right)\\
&=\frac{\prod_{j=0}^{h_i-1}g\left(s+\dfrac{w_i}{dh_i}+\dfrac{j}{h_i}\right)}{\chi_{-(h_is+\frac{w_i}{d})}(h_i)\prod_{j=1}^{h_i-1}g\left(\dfrac{j}{h_i}\right)}\end{align*}

Combining, we get that for a fixed $w$,

\[ (\ref{E:calculation})=\frac{c}{q-1}\sum_{s\in\frac{1}{q-1}\Z/\Z}\frac{\prod_{i=1}^n\prod_{j=0}^{h_i-1}g\left(s+\dfrac{w_i+dj}{dh_i}\right)}{g(s)g\left(s+\dfrac{1}{d}\right)\cdots g\left(s+\dfrac{d-1}{d}\right)}\chi_s\left(\prod_ih_i^{h_i}\lambda^d\right),\] where

\[c=\frac{\prod_{j=1}^{d-1}g\left(\dfrac{j}{d}\right)}{\prod_{i=1}^n\prod_{j=1}^{h_i-1}g\left(\dfrac{j}{h_i}\right)}.\]

Notice that over $\frac{1}{q-1}\Z/\Z$, $g(-s)=g(1-s)$, and so property \ref{L:Gauss}.1 of Gauss sums can be rewritten as

\[g(s)g(1-s)=q\chi_s(-1).\]

Using this, we can rewrite the products above as

\[\prod_{j=1}^{d-1}g\left(\dfrac{j}{d}\right)=q^{\frac{d-1}{2}}\xi_1\] and

\[\prod_{i=1}^n\prod_{j=1}^{h_i-1}g\left(\dfrac{j}{h_i}\right)=q^{\sum\frac{h_i-1}{2}}\xi_2=q^{\frac{d-n}{2}}\xi_2,\] where $\xi_1,\xi_2$ are $q-1$ roots of unity. And so $c$ becomes much simpler:

\[c=\xi q^{\frac{n-1}{2}},\] where $\xi$ is some root of unity which depends on $d,n, h_i$.

We want to relate this last expression to Katz's hypergeometric function. Notice that it is almost in the same form as Lemma \ref{L:katz}, except that we need to add over $\overline{\chi_s}=\chi_{-s}$, but we can change variables in the sum, so that we get

\[\frac{\xi q^{\frac{n-1}{2}}}{q-1}\sum_{s\in\frac{1}{q-1}\Z/\Z}\frac{\prod_{i=1}^n\prod_{j=0}^{h_i-1}g\left(-s+\dfrac{w_i+dj}{dh_i}\right)}{g(-s)g\left(-s+\dfrac{1}{d}\right)\cdots g\left(-s+\dfrac{d-1}{d}\right)}\chi_{-s}\left(\prod_ih_i^{h_i}\lambda^d\right)\]
\[
=\frac{\xi q^{\frac{n-1}{2}}}{q-1}\sum_{s\in\frac{1}{q-1}\Z/\Z}\frac{\prod_{i=1}^n\prod_{j=0}^{h_i-1}g\left(-\left(s-\dfrac{w_i+dj}{dh_i}\right)\right)}{g(-s)\cdots g\left(-\left(s-\dfrac{d-1}{d}\right)\right)}\overline{\chi}_s\left(\prod_ih_i^{h_i}\lambda^d\right).\]

Now we can use property \ref{L:Gauss}.1 of Gauss sums to change from expressions involving $g(-s)$ to expressions involving $g(s)$ and viceversa by

\[g(-s)=\frac{q\chi_s(-1)}{g(s)},\] to get

\begin{align*}
&\frac{c'}{q-1}\sum_{s\in\frac{1}{q-1}\Z/\Z}g(s)g\left(s-\dfrac{1}{d}\right)\cdots g\left(s-\dfrac{d-1}{d}\right)\\
&\cdot\prod_{i=1}^n\prod_{j=0}^{h_i-1}g\left(-\left(s-\dfrac{w_i+dj}{dh_i}\right)\right)\chi_{-s}\left((-1)^d\prod_ih_i^{h_i}\lambda^d\right)\\
&=\frac{c'}{q-1}\sum_{s\in\frac{1}{q-1}\Z/\Z}g(s)g\left(s+1-\dfrac{1}{d}\right)\cdots g\left(s+1-\dfrac{d-1}{d}\right)\\
&\cdot\prod_{i=1}^n\prod_{j=0}^{h_i-1}g\left(-\left(s+1-\dfrac{w_i+dj}{dh_i}\right)\right)\chi_{-s}\left((-1)^d\prod_ih_i^{h_i}\lambda^d\right)\\
&=c{'}H\left(\left.0,\frac{1}{d},\dots,\frac{d-1}{d};\dots,1-\frac{w_i+dj}{dh_i},\dots\right|\prod_{i=1}^nh_i^{h_i}(-\lambda)^d\right),
\end{align*} where the denominator parameters run through the $h_i$ values $\frac{w_i+dj}{dh_i}, j=0,\dots,h_i-1$ for each $i$, and no exponent appears if $h_j=0$, and modulo some cancelation if some of the numerator and denominator terms are the same. Notice that there will be the same number of upper and lower exponents, since we required $\sum h_i=d$. The constant term is now

\[c{'}= \xi q^{\frac{n-1}{2}}\cdot\frac{\chi_{\frac{d-1}{2}}(-1)}{q^d}=\xi q^{\frac{n-2d-1}{2}},\] where $\xi$ still denotes a $q-1$ root of unity.

To get the total number of points we would need to add over equivalence class representatives, and so

\begin{align*}
&N_{\F_q}(\lambda)-N_{\F_q}(0)=\xi q^{\frac{n-2d-1}{2}}\cdot\\
&\sum_{[w]\in W/\sim}H\left(\left.0,\frac{1}{d},\dots,\frac{d-1}{d};\dots,1-\frac{w_i+dj}{dh_i},\dots\right|\prod_{i=1}^nh_i^{h_i}(-\lambda)^d\right).
\end{align*}\qedhere

\end{proof}

\begin{rmk}
Notice that the above formula implies that the hypergeometric function is independent of the choice of representative $w$. This is because the characters that define $H$ were defined modulo integer powers, and $w'\sim w$ means that $w_i'\equiv w_i+kh_i\bmod d$, so substituting by an equivalent $w$ gives the same characters for $H$.
\end{rmk}

\subsection{A $0$-dimensional example}

The most basic example of a family like (\ref{E:monomial}) is the 0-dimensional family defined by

\[Z_{\lambda}: x_1^d+x_2^d-d\lambda x_1x_2^{d-1}=0.\]

Notice that to put this in the situation of Koblitz's theorem in the previous section, we have to assume $d(d-1)|q-1$, and we have $h=(1,d-1)$. Also, we can see that $W=\{(0,0), (1,d-1), \dots, (d-1,1)\}$, so in particular there is only one equivalence class, that of $(0,0)$. So using the last equation, we get that

\begin{align*}
&N_{\F_q}(\lambda)-N_{\F_q}(0)=\xi q^{\frac{3-2d}{2}}\cdot\\
&H\left(\left.\frac{1}{d},\dots,\frac{d-1}{d};0,\frac{1}{d-1},\dots,\frac{d-2}{d-1}\right|-(d-1)^{(d-1)}(-\lambda)^d\right).\\
\end{align*}

In the case $d=3$, the number of points is

\[N_{\F_q}(\lambda)-N_{\F_q}(0)= \xi q^{-\frac{3}{2}}H\left(\left.\frac{1}{3},\frac{2}{3};0,\frac{1}{2}\right|2^2\lambda^3\right).\]

\subsection{The Dwork family}
 The Dwork family is a family of the type (\ref{E:monomial}) with $n=d$ and $h_i=1$ for all $i$. That is, the family

    \[Y_{\lambda}:x_1^d+\cdots+x_d^d-d\lambda x_1\cdots x_d=0.\]

    The cases $d=3,4$ were studied extensively by Dwork while he was studying the rationality of the Zeta function, for example in \cite{dwork:padic}.

    In this case, for each equivalence class we get that
\begin{align*}
&\sum_{s\in\frac{1}{q-1}\Z/\Z}\frac{g\left(hs+\dfrac{w}{d}\right)}{g(ds)}\chi_{ds}(d\lambda)\\
&=\prod_{j=1}^{d-1} g\left(\frac{j}{d}\right)\sum_s\frac{g\left(h_1s+\dfrac{w_1}{d}\right)\cdots g\left(h_ns+\dfrac{w_n}{d}\right)}{g(s)g\left(s+\dfrac{1}{d}\right)\cdots g\left(s+\dfrac{d-1}{d}\right)}\chi_{ds}(\lambda)\\
&=\prod_{j=1}^{d-1} g\left(\frac{j}{d}\right)\sum_{s\in\frac{1}{q-1}\Z/\Z}\frac{g\left(s+\dfrac{w_1}{d}\right)\cdots g\left(s+\dfrac{w_n}{d}\right)}{g(s)g\left(s+\dfrac{1}{d}\right)\cdots g\left(s+\dfrac{d-1}{d}\right)}\chi_{s}(\lambda^d).\\
\end{align*}

There will be cancelation when the $w_i$ coincide with $0,1,\dots,d-1$. Again, we replace $s$ by $-s$ and get that

\begin{align*}
&N_{\F_q}(\lambda)-N_{\F_q}(0)=\xi q^{\frac{-d-1}{2}}\cdot\\
&\sum_{[w]\in W/\sim}H\left(\left.0,\frac{1}{d},\dots,\frac{d-1}{d};1-\frac{w_1}{d},\dots,1-\frac{w_n}{d}\right|(-\lambda)^d\right).
\end{align*}

Let $d=3$ (the family is actually a family of elliptic curves). In other words, the family with $d=3=n, h=(1,1,1)$.

We can see that $$W=\{(0,0,0), (1,1,1), (2,2,2), (1,2,0), (2,0,1),$$ $$(0,1,2), (2,1,0), (0,2,1), (1,0,2)\}.$$ And, in fact, there are three equivalence class representatives, $(0,0,0)$, $(1,2,0)$, $(2,1,0)$, but the latter two are of the same ``type'', i.e., one is the permutation of the other. Therefore, we obtain

\[N_{\F_q}(\lambda)-N_{\F_q}(0)=\xi q^{-1}H\left(\left.\frac{1}{3},\frac{2}{3};1,1\right|\lambda^3\right)+ \frac{2\xi q}{(q-1)}\sum_{s\in\frac{1}{q-1}\Z/\Z}\overline{\chi}_s(\lambda^3)\]

For the terms corresponding to the ``type'' $(1,2,0)$, the $w_i$'s completely cancel out with the list $0,1,2$, which means we have an empty parameter set. This also means that $H$ is the sum over all multiplicative characters of $\chi_s(\lambda^3)$, for $\lambda\in\F_q^*$, which is zero unless $\lambda^3=1$ in $\F_q^*$, in which case we get $(q-1)$.

\begin{rmk}
It is not difficult to check  that the $\lambda$'s that make $Y_{\lambda}$ singular are exactly the  $d$-th roots of unity. And so for all $\lambda$ such that $Y_{\lambda}$ is non-singular, the second term in the above sum is zero, and we get that the number of points is written in terms of a hypergeometric function.
\end{rmk}

\section{$p$-adic methods}

The main goal of this section is to develop a $p$-adic version of Koblitz's formula for $N_{\F_p}(\lambda)$, where $p$ is prime, so that we can find the relation between the number of solutions over $\F_p$ and generalized hypergeometric functions. We will first summarize the main ideas of the Gross-Koblitz formula, and then restrict our attention to two special examples.

\subsection{The Gross-Koblitz Formula}

The Gross-Koblitz formula was developed as a way of relating Gauss sums to the $p$-adic version of the $\Gamma$ function. For a more detailed account, see \cite{robert}.

First, we will need to recall the following definition by Morita:

\begin{defn}
The \emph{$p$-adic gamma function} is the continuous function

\[\Gamma_p:\Z_p\rightarrow\Z_p\]

\noindent that extends

\[f(n):=(-1)^n\prod_{1\leq j<n,p\not\mid j}j \hspace{1cm} (n\geq 2).\]

\end{defn}

This function has properties that are reminiscent of those of the classical gamma function.

\begin{prop}

Let $p$ be an odd prime.

\begin{enumerate}
\item $\Gamma_p(0)=1, \hspace{.5cm} \Gamma_p(1)=-1, \hspace{.5cm} \Gamma_p(2)=1, \hspace{.5cm}$

$\Gamma_p(n+1)=(-1)^{n+1}n! \hspace{.5cm}(1\leq n<p).$

\item $\Gamma_p(x+1)=\left\{\begin{array}{cc} -x\Gamma_p(x)& \text{if $x\in\Z_p^{*}$},\\
                                                -\Gamma_p(x)& \text{if $x\in p\Z_p$}
                                                \end{array}\right.$
\item $\Gamma_p(x)\Gamma_p(1-x)=(-1)^{R(x)}$, where $R(x)\in\{1,2,\dots,p\}$, $R(x)\equiv x\bmod p$.

\item (Gauss multiplication formula) Let $m\geq 1$ be an integer prime to $p$. Then

\[\prod_{0\leq j<m} \Gamma_p\left(x+\frac{j}{m}\right)=\epsilon_m\cdot m^{1-R(mx)}\cdot(m^{p-1})^{s(mx)}\cdot\Gamma_p(mx),\]

\noindent where

\[\epsilon_m=\prod_{0\leq j<m}\Gamma_p\left(\frac{j}{m}\right),\]
\[R(y)\in\{1,\dots,p\}, R(y)\equiv y\bmod p,\]
\[s(y)=\frac{R(y)-y}{p}\in\Z_p.\]

\end{enumerate}
\label{P:gamma}
\end{prop}

Let $s=\dfrac{a}{q-1}\in\dfrac{1}{q-1}\Z/\Z$, $\omega$ be the Teichm\"{u}ller character and $\psi$ be an additive character of $\F_q$, as before. Consider now the Gauss sum
\[g(s)=\sum_{0\neq x\in\F_q}\omega(x)^{-s(q-1)}\psi(x),\]

 Suppose $q=p^f$. Let $\pi\in\C_p$ be a root of $\pi^{p-1}=-p$. Define for $0\leq \dfrac{a}{q-1}=s<1$, the sum $S_p(a)=\sum_{0\leq j<f}a_j$ to be the sum of the digits in the $p$-adic expansion of $a$, and the integers $a^{(i)}$ as having $p$-adic expansions obtained from the cyclic permutations from the expansion of $a$ (denoted $a^{(0)}$).

\begin{theorem}[Gross-Koblitz]
   Let $0\leq s=\dfrac{a}{q-1}<1$. The value of the Gauss sum $g(s)$ is explicitely given by

\[g(s)=-\pi^{S_p(a)}\prod_{0\leq j<f}\Gamma_p\left(\frac{a^{(j)}}{q-1}\right).\]

\end{theorem}

For a proof of this theorem see \cite{robgamma}.

Over $\F_p$, i.e. if we assume $f=1$, the formula becomes much simpler, yielding

\begin{equation}g(s)=-\pi^a\Gamma_p\left(\frac{a}{p-1}\right)=-\pi^{s(p-1)}\Gamma_p(s)=-(-p)^s\Gamma_p(s)\label{E:GK}\end{equation}

We will use this theorem to produce precise formulas which  will show the relation between the number of points and generalized hypergeometric functions, focusing our attention on $\F_p$.

\subsection{The 0-dimensional example}

As seen in the previous section, the easiest example to deal with is the $0$-dimensional variety

 \[Z_{\lambda}: x_1^d+x_2^d-d\lambda x_1x_2^{d-1}=0.\]

 Recall that Koblitz's theorem gives, in this case, that

\[N_{\F_p}(\lambda)-N_{\F_p}(0)= \frac{1}{p-1}\sum_{s\in\frac{1}{p-1}\Z/\Z}\frac{g(s)g((d-1)s)}{g(ds)}\chi_{ds}(\lambda).\]

Assume that the generator of the multiplicative character group is $\omega^{-1}$ and $p$ is a prime such that $d|p-1$. Using the Gross-Koblitz formula yields the following formula. 

\begin{lemma}

 \begin{align*}& N_{\F_p}(\lambda)=N_{\F_p}(0)\\
 &+\frac{-1}{p-1}\sum_{a=0}^{p-2}\frac{(-p)^{\eta(a)}\Gamma_p\left(\dfrac{a}{p-1}\right)\Gamma_p\left(\left\{\dfrac{(d-1)a}{p-1}\right\}\right)}{\Gamma_p\left(\left\{\dfrac{da}{p-1}\right\}\right)}\omega(d\lambda)^{-da}
 \end{align*} where $\eta(a)=\displaystyle\left(\frac{a}{p-1}+\left\{\frac{(d-1)a}{p-1}\right\}-\left\{\frac{da}{p-1}\right\}\right)$.
\label{L:mine}
\end{lemma}

\begin{proof}

First, notice that we can rewrite $N_{\F_p}(\lambda)-N_{\F_p}(0)$ by changing its summation indices as follows:

\begin{align*}&\sum_{s\in\frac{1}{p-1}\Z/\Z}\frac{g(s)g((d-1)s)}{g(ds)}\omega(d\lambda)^{-ds(p-1)}\\
&=\sum_{a=0}^{p-2}\frac{g\left(\left\{\dfrac{(d-1)a}{p-1}\right\}\right)g\left(\dfrac{a}{p-1}\right)}{g\left(\left\{\dfrac{da}{p-1}\right\}\right)}\omega(d\lambda)^{-da}.\end{align*}

Using (\ref{E:GK}), we get

\begin{align*}
&\sum_{a=0}^{p-2}\frac{g\left(\left\{\dfrac{(d-1)a}{p-1}\right\}\right)g\left(\dfrac{a}{p-1}\right)}{g\left(\left\{\dfrac{da}{p-1}\right\}\right)}\omega(d\lambda)^{-da}\\
&=\sum_{a=0}^{p-2}\frac{(-p)^{\left(\left\{\frac{(d-1)a}{p-1}\right\}\right)}\Gamma_p\left(\left\{\dfrac{(d-1)a}{p-1}\right\}\right)(-p)^{\left(\frac{a}{p-1}\right)}\Gamma_p\left(\dfrac{a}{p-1}\right)}{(-p)^{\left(\left\{\frac{da}{p-1}\right\}\right)}\Gamma_p\left(\left\{\dfrac{da}{p-1}\right\}\right)}\omega(d\lambda)^{-da}\\
&=\sum_{a=0}^{p-2}\frac{(-p)^{(\frac{a}{p-1}+\{\frac{(d-1)a}{p-1}\}-\{\frac{da}{p-1}\})}\Gamma_p\left(\dfrac{a}{p-1}\right)\Gamma_p\left(\left\{\dfrac{(d-1)a}{p-1}\right\}\right)}{\Gamma_p\left(\left\{\dfrac{da}{p-1}\right\}\right)}\omega(d\lambda)^{-da}\end{align*}

And, thus 

\begin{align*}
&N_{\F_p}(\lambda)-N_{\F_p}(0)\\
&=\frac{-1}{p-1}\sum_{a=0}^{p-2}\frac{(-p)^{\eta(a)}\Gamma_p\left(\dfrac{a}{p-1}\right)\Gamma_p\left(\left\{\dfrac{(d-1)a}{p-1}\right\}\right)}{\Gamma_p\left(\left\{\dfrac{da}{p-1}\right\}\right)}\omega(d\lambda)^{-da}
\end{align*}\qedhere
\end{proof}

Suppose the have a hypergeometric weight system (see Section \ref{S:weight}) given by $\gamma=[d]-[1]-[d-1]$. This is related to the power series with binomial coefficients $dn\choose n$. The Landau function associated to this system is

\[\mathcal{L}(x)=\{x\}+\{(d-1)x\}-\{dx\}.\]

Notice that the power of $p$ that appears in Lemma \ref{L:mine} is exactly determined by $\mathcal{L}\left(\dfrac{a}{p-1}\right)$. But this means that the valuation of the terms of the previous sum is very similar to the valuation of the terms in the hypergeometric series with coefficients $dn \choose n$.

Notice that

\[\sum_{n\geq0}{dn\choose n}z^n={}_{d-1}F_{d-2}\left(\left.\frac{1}{d},\dots,\frac{d-1}{d};\frac{1}{d-1},\dots,\frac{d-1}{d-2}\right|\frac{(d-1)^{(d-1)}}{d^d}z\right).\]

The discontinuities of $\mathcal{L}$ are therefore the $\alpha_i$ and $\beta_i$ parameters. In fact, it is clear that the parameters interlace, that is, $0<\frac{1}{d}<\frac{1}{d-1}<\cdots<\frac{d-2}{d-1}<\frac{d-1}{d}<1$. By property \ref{P:Landau}.1 of the Landau function we get that $\mathcal{L}\left(\dfrac{a}{p-1}\right)=1$ for $(p-1)\alpha_i\leq a<(p-1)\beta_i$ and 0 on the other intervals. Therefore the terms with $(p-1)\beta_i=\frac{(p-1)i}{d-1}\leq a<(p-1)\alpha_{i+1}=\frac{(p-1)(i+1)}{d}$ are the only ones that survive mod $p$. There are $d-1$ of these intervals. For a fixed $i$,

\begin{align*}
&\sum_{a=\frac{(p-1)i}{d-1}}^{\frac{(p-1)(i+1)}{d}-1}\frac{(-p)^{(\frac{a}{p-1}+\{\frac{(d-1)a}{p-1}\}-\{\frac{da}{p-1}\})}\Gamma_p\left(\dfrac{a}{p-1}\right)\Gamma_p\left(\left\{\dfrac{(d-1)a}{p-1}\right\}\right)}{\Gamma_p\left(\left\{\dfrac{da}{p-1}\right\}\right)}\omega(d\lambda)^{-da}\\
&=\sum_{a=\frac{(p-1)i}{d-1}}^{\frac{(p-1)(i+1)}{d}-1}\frac{\Gamma_p\left(\dfrac{a}{p-1}\right)\Gamma_p\left(\left\{\dfrac{(d-1)a}{p-1}\right\}\right)}{\Gamma_p\left(\left\{\dfrac{da}{p-1}\right\}\right)}\omega(d\lambda)^{-da}\\
&=\sum_{a=\frac{(p-1)i}{d-1}}^{\frac{(p-1)(i+1)}{d}-1}\frac{\Gamma_p\left(\dfrac{a}{p-1}\right)\Gamma_p\left(\dfrac{(d-1)a}{p-1}-i\right)}{\Gamma_p\left(\dfrac{da}{p-1}-i\right)}\omega(d\lambda)^{-da}\\
&\equiv\sum_{a=\frac{(p-1)i}{d-1}}^{\frac{(p-1)(i+1)}{d}-1}\frac{\Gamma_p(-a)\Gamma_p(-(d-1)a-i)}{\Gamma_p(-da-i)}(d\lambda)^{-da}\bmod p\\
\end{align*} And if we now use property \ref{P:gamma}.3 of the $p$-adic gamma function, we get 
\begin{align*}
&\equiv\sum_{a=\frac{(p-1)i}{d-1}}^{\frac{(p-1)(i+1)}{d}-1}\frac{\Gamma_p(da+i+1)}{\Gamma_p(a+1)\Gamma_p((d-1)a+i+1)}(d\lambda)^{-da}\bmod p \hspace{1cm}\\
\end{align*} And by property \ref{P:gamma}.1
\begin{align*}
&\equiv\sum_{a=\frac{(p-1)i}{d-1}}^{\frac{(p-1)(i+1)}{d}-1}\frac{(da+i)!}{a!((d-1)a+i)!}(d\lambda)^{-da}\bmod p\\
&\equiv\sum_{a=\frac{(p-1)i}{d-1}}^{\frac{(p-1)(i+1)}{d}-1}{{da+i}\choose a}(d\lambda)^{-da}\bmod p.
\end{align*}

And we have just shown that

\[N_{\F_p}(\lambda)-N_{\F_p}(0)\equiv\sum_{i=0}^{d-2} \sum_{a=(p-1)\beta_i}^{(p-1)\alpha_{i+1}-1}{{da+i}\choose a}(d\lambda)^{-da}\bmod p.\]

Notice that, for a fixed $i$,

\begin{align*}
&{{da+i}\choose a}=\frac{(da+i)(da+i-1)\cdots(da+1)}{((d-1)a+i)((d-1)a+i-1)\cdots((d-1)a+1)}{da\choose a}\\
&=\frac{(da+i)\cdots(da+1)}{((d-1)a+i)\cdots((d-1)a+1)}\cdot\frac{d^{da}(\frac{1}{d})_a\cdots(\frac{d-1}{d})_a}{(d-1)^{(d-1)a}a!(\frac{1}{d-1})_a\cdots(\frac{d-2}{d-1})_a}\\
&=\frac{d^{i}(a+\frac{i}{d})\cdots(a+\frac{1}{d})}{(d-1)^{i}(a+\frac{i}{d-1})\cdots(a+\frac{1}{d-1})}\cdot\frac{d^{da}(\frac{1}{d})_a\cdots(\frac{d-1}{d})_a}{(d-1)^{(d-1)a}a!(\frac{1}{d-1})_a\cdots(\frac{d-2}{d-1})_a},\\
\end{align*} and we can combine the products so that the last expression equals

$$=\frac{d^i\frac{1}{d}\cdots\frac{i}{d}}{(d-1)^i\frac{1}{d-1}\cdots\frac{i}{d-1}}\cdot
\frac{d^{da}(\frac{1}{d}+1)_a\cdots(\frac{i}{d}+1)_a(\frac{i+1}{d})_a\cdots(\frac{d-1}{d})_a}{(d-1)^{(d-1)a}(\frac{1}{d-1}+1)_a\cdots(\frac{i}{d-1}+1)_a(\frac{i+1}{d-1})_a\cdots(\frac{d-2}{d-1})_a}.$$

We have just proved:

\begin{theorem}
Let $\alpha^{(0)}=\left(\frac{1}{d},\dots,\frac{d-1}{d}\right)$ and $\beta^{(0)}=\left(\frac{1}{d-1},\dots,\frac{d-2}{d-1}\right)$.

$$N_{\F_p}(\lambda)-N_{\F_p}(0)\equiv\sum_{i=0}^{d-2}\left[{}_{d}F_{d-1}(\alpha^{(i)};\beta^{(i)}|(d-1)^{-(d-1)}\lambda^{-d})\right]_{\frac{i(p-1)}{d-1}}^{\frac{(i+1)(p-1)}{d}-1}\bmod p,$$

\noindent where $\alpha^{(i)}=(\alpha_1+1,\dots,\alpha_i+1,\alpha_{i+1},\dots,\alpha_{d-1})$, and $\beta^{(i)}=(\beta_1+1,\dots,\beta_i+1,\beta_{i+1},\dots,\beta_{d-2})$, that is we add 1 to the numerator and denominator parameters up to the $i$-th place.
\label{T:mine}
 \end{theorem}

\begin{notation}
$[(u(z)]_i^j$ denotes the polynomial which is the truncation of a series $u(z)$ from $n=i$ to $j$.
\end{notation}

So for example in the case $d=3$ we get that
\begin{align*}
N_{\F_p}(\lambda)-N_{\F_p}(0)\equiv&\left[{}_2F_1\left(\left.\frac{1}{3},\frac{2}{3};\frac{1}{2}\right|\frac{1}{2^2\lambda^3}\right)\right]_0^{\frac{p-1}{3}-1}\\
&+\left[{}_2F_1\left(\left.\frac{4}{3},\frac{2}{3};\frac{3}{2}\right|\frac{1}{2^2\lambda^3}\right)\right]_{\frac{p-1}{2}}^{\frac{2(p-1)}{3}-1}\bmod p.\\
\end{align*}

\begin{rmk}
Notice the difference between Theorem 5.4 and Igusa's result: in our situation, more than one hypergeometric function appears. As far as we know, most of the known examples that have been computed have coincided with Igusa in the sense that only one hypergeometric function appears modulo $p$. In the next two examples (the Dwork family), we will show a known computation using our methods, in which only one hypergeometric function appears.
\end{rmk}

\subsection{The Dwork family when $d=3$ (The elliptic curve case)}

Recall that we have three equivalence classes, $(0,0,0),(1,2,0),(2,1,0)$, and so we can split the sum into three sums (although since the last two are permutations of each other the sums will be the same), so we get:

\begin{align*}
N_{\F_p}(\lambda)-N_{\F_p}(0)=&\frac{1}{p-1}\sum_{s\in\frac{1}{p-1}\Z/\Z}\frac{g(s)^3}{g(3s)}\chi_{3s}(3\lambda)\\
& +\frac{2}{p-1}\sum_{s\in\frac{1}{p-1}\Z/\Z}\frac{g(s)g\left(s+\frac{1}{3}\right)g\left(s+\frac{2}{3}\right)}{g(3s)}\chi_{3s}(3\lambda)
\end{align*}

Again, before using the formula it is convenient to change the summation:

\begin{align*}&N_{\F_p}(\lambda)-N_{\F_p}(0)=\frac{1}{p-1}\sum_{s=0}^{p-2}\frac{g\left(\frac{s}{p-1}\right)^3}{g\left(\left\{\frac{3s}{p-1}\right\}\right)}\omega(3\lambda)^{-3s}
\\
&+\frac{2}{p-1}\sum_{s=0}^{p-2}\frac{g\left(\frac{s}{p-1}\right)g\left(\left\{\frac{s}{p-1}+\frac{1}{3}\right\}\right)g\left(\left\{\frac{s}{p-1}+\frac{2}{3}\right\}\right)}{g\left(\left\{\frac{3s}{p-1}\right\}\right)}\omega(3\lambda)^{-3s}\end{align*}

Substituting and simplifying, we get that
\begin{align*}
&N_{\F_p}(\lambda)-N_{\F_p}(0)=\frac{1}{p-1}\sum_{s=0}^{p-2}\frac{(-p)^{\left(\frac{3s}{p-1}-\left\{\frac{3s}{p-1}\right\}\right)}
\Gamma_p\left(\frac{s}{p-1}\right)^3}{\Gamma_p\left(\left\{\frac{3s}{p-1}\right\}\right)}\omega(3\lambda)^{-3s}+\\
&+\frac{2}{p-1}\sum_{s=0}^{p-2}\frac{(-p)^{\gamma(s)}
\Gamma_p\left(\frac{s}{p-1}\right)\Gamma_p\left(\left\{\frac{s}{p-1}+\frac{1}{3}\right\}\right)\Gamma_p\left(\left\{\frac{s}{p-1}+\frac{2}{3}\right\}\right)}
{\Gamma_p\left(\left\{\frac{3s}{p-1}\right\}\right)}\omega(3\lambda)^{-3s},
\end{align*} where $\gamma(s)={\left(\frac{s}{p-1}+\left\{\frac{s}{p-1}+\frac{1}{3}\right\}+\left\{\frac{s}{p-1}+\frac{2}{3}\right\}-\left\{\frac{3s}{p-1}\right\}\right)}$.

 Once more, the power of $p$ in the first part of the sum is determined by $\mathcal{L}\left(\frac{s}{p-1}\right)$ where $\mathcal{L}(x)$ is the Landau function associated to the hypergeometric weight system $[3]-3[1]$. The discontinuities are $0,1/3,2/3$ and the function is zero only when $0\leq x<1/3$. So $\bmod p$ we get

 \[N_{\F_p}(\lambda)-N_{\F_p}(0)\equiv-\sum_{s=0}^{\frac{p-1}{3}-1}\frac{\Gamma_p(-s)^3}{\Gamma_p(-3s)}(3\lambda)^{-3s}\bmod p\]
 \[\equiv -\sum_{s=0}^{\frac{p-1}{3}-1}\frac{\Gamma_p(1+3s)}{\Gamma_p(1+s)^3}(3\lambda)^{-3s}\bmod p\]
\[\equiv -\sum_{s=0}^{\frac{p-1}{3}-1}\frac{(3s)!}{s!^3}(3\lambda)^{-3s}\bmod p\]
\[\equiv-\left[{}_2F_1\left(\left.\frac{1}{3},\frac{2}{3};1\right|\lambda^{-3}\right)\right]_0^{\frac{p-1}{3}-1}\bmod p.\]

\subsection{The Dwork family when $d=4$ (The $K3$-surface case)}

This is the case

\[X_{\lambda}:x_1^4+x_2^4+x_3^4+x_4^4-4\lambda x_1x_2x_3x_4=0\]

In other words, the family with $d=4=n, h=(1,1,1,1)$.

The set $W$ is made up of 64 vectors, but we can split them up into 16 orbits, and of those there are only three "types". These are
\[(0,0,0,0),(1,1,1,1),(2,2,2,2),(3,3,3,3)\]
\[(0,1,1,2),(1,2,2,3),(2,3,3,0),(3,0,0,1)\]
\[(0,0,2,2),(1,1,3,3),(2,2,0,0),(3,3,1,1)\]

The rest are permutations of these. So there is one orbit of the first type, 12 orbits of the second type, and 3 orbits of the third type.

This makes the formula look as follows:

\[N_{\F_p}(\lambda)-N_{\F_p}(0)=\frac{1}{p-1}\sum_{s\in\frac{1}{p-1}\Z/\Z}\frac{g(s)^4}{g(4s)}\chi_{4s}(4\lambda)\hspace{1cm} (S_1)\]
\[+\frac{12}{p-1}\sum_{s\in\frac{1}{p-1}\Z/\Z}\frac{g(s)g(s+\frac{1}{4})^2g(s+\frac{1}{2})}{g(4s)}\chi_{4s}(4\lambda)\hspace{1cm} (S_2)\]
\[+\frac{3}{p-1}\sum_{s\in\frac{1}{p-1}\Z/\Z}\frac{g(s)^2g(s+\frac{1}{2})^2}{g(4s)}\chi_{4s}(4\lambda).\hspace{1cm} (S_3)\]

Let's focus on the first term of the sum, denoted by $S_1$. Using Gross-Koblitz we get

\[S_1=\frac{1}{p-1}\sum_{s=0}^{p-2}\frac{(-p)^{(\frac{4s}{p-1}-\left\{\frac{4s}{p-1}\right\})}\Gamma_p\left(\frac{s}{p-1}\right)^4}{\Gamma_p\left(\left\{\frac{4s}{p-1}\right\}\right)}\omega(4\lambda)^{-4s}\]

By inspecting the power of $-p$ we can see that again it is determined by $\mathcal{L}_{\gamma}$ where $\gamma=[4]-4[1]$. Thus, the only terms that survive mod $p$ are those for which $0\leq s<\frac{p-1}{4}$. So

\[S_1\equiv-\sum_{s=0}^{\frac{p-1}{4}-1}\frac{\Gamma_p(-s)^4}{\Gamma_p(-4s)}(4\lambda)^{-4s}\bmod p\]

\[\equiv \sum_{s=0}^{\frac{p-1}{4}-1}\frac{\Gamma_p(1+4s)}{\Gamma_p(1+s)^4}(4\lambda)^{-4s}\bmod p\]
\[\equiv \sum_{s=0}^{\frac{p-1}{4}-1} \frac{(4s)!}{(s!)^4}(4\lambda)^{-4s}\bmod p\]
\[\equiv \left[{}_3F_2\left(\left.\frac{1}{4},\frac{1}{2},\frac{3}{4};1,1\right|\lambda^{-4}\right)\right]_0^{\frac{p-1}{4}-1}\bmod p.\]

Inspection shows that $S_2$ and $S_3$ are both zero modulo $p$.

\begin{rmk}Notice that in both the $d=3$ and $d=4$ examples, the only terms to survive mod $p$ are the ones related to the class of $(0,\dots,0)$. Clearly some information is lost that might not be lost if we studied these cases modulo other powers of $p$. One of our future plans is to try using the Gross-Koblitz formula for the more general finite fields to compute these examples. In the case of the elliptic curve, we believe $p^3$ will be the right power, and we expect that for any $d$, we should study the number of solutions modulo $p^d$. This was actually checked by Rodr\'{i}guez-Villegas, Candelas and de la Ossa for $d=5$ in \cite{frv:candelas2}.
\end{rmk}

\subsection*{Acknowledgements:} This paper is based on some results obtained in the author's Ph.D. thesis. As such, the author would primarily like to thank her thesis advisor, Fernando Rodr\'{i}guez-Villegas, for his guidance, support and great ideas. The author would also like to thank, in no particular order, Bjorn Poonen, Benjamin Brubaker, and Catherine Lennon for many valuable conversations at the time this paper was being written. 

%\bibliographystyle[12pt]{siam} 
%\bibliography{thesis} 

\end{document}